\documentclass[a4paper,12pt]{article}
\usepackage[a4paper,left=25mm,right=25mm,top=30mm,bottom=30mm]{geometry}
\usepackage{amsmath,  amsthm, amssymb}
\usepackage{graphicx} 
\numberwithin{equation}{section}

\newtheorem{lemma}{Lemma}

\begin{document}

\author{Ajai Choudhry}
\title{Construction of parametrized families of \\ general pseudo-bimagic and bimagic squares}

\date{}
\maketitle
\begin{abstract} A  magic square of order $n$ is a square array of $n^2$ distinct integers such that the integers of every row, every column  and  the two diagonals have the same common sum. If a magic  square remains a magic square even when each  integer  is replaced by its square, it is called bimagic. If, however, all the rows and all the columns of the aforementioned  square of squares add up to a common sum, but the two diagonals do not, the original magic square is called pseudo-bimagic. Except for one family of $6 \times 6$ pseudo-bimagic squares,  parametrized families of bimagic squares have not been published.  In this paper we construct  parametrized families of $6 \times 6$ pseudo-bimagic squares and $8 \times 8$ bimagic squares. We  use these families to construct, for any positive integer $m$,  multi-parameter families of  pseudo-bimagic squares of order $6m$  and bimagic squares of order $8m$. We also give  numerical examples of $6 \times 6$ and $12 \times 12$ pseudo-bimagic squares and an $8 \times 8$  bimagic square. 
\end{abstract}

\medskip

\medskip

\section{Introduction}\label{intro}

A magic square of order $n$ is a square array of $n^2$ distinct integers such that the sums of the integers in each row and each column  as well as each of the two diagonals are all equal. If the sums for only the rows and the columns are equal, and the integers of the  diagonals do not add up to the same common sum, the square is said to be semi-magic. If  a magic square is made up of the first $n^2$ consecutive integers $1, 2, \ldots, n^2$, the square is said to be a  normal magic square, otherwise, it is called a general magic square. A bimagic square is a magic square that remains magic when all of its integers are replaced by their squares. A semi-bimagic square is a semi-magic square that remains semi-magic when all of its integers are replaced by their squares. 
Finally, we define a {\it pseudo-bimagic square} as a magic square that becomes a semi-magic square when all of its integers  are replaced by their squares. In each case, the common sum of the integers of each row, each column and each diagonal (if applicable) is called the magic sum.

While the construction of magic squares has a long history as may be seen in books by Andrews \cite{An} and Kraitchik \cite{Kr}, this paper is concerned with  bimagic squares. Considerable work has been done on the construction of normal bimagic squares. Pfeffermann constructed  a normal bimagic square of order 8  in 1890 (as quoted by Boyer \cite{Bo}), and since then numerous examples of normal bimagic squares of order $n \geq 8$ have been constructed by several mathematicians using various methods (\cite[Chapter 12, pp. 78--83 and Chapter 28, pp. 175--179]{BJ}, \cite{Bo}, \cite{CL}, \cite{CLP}, \cite{FG1}, \cite{FG2},
 \cite[pp. 177--78]{Kr} \cite{LPC}, \cite{Ta}, \cite{TG}). It is now known that a normal bimagic square of order $n$ does not exist when $ 2 \leq n \leq 7$ \cite{Bo2}.  Recently it was proved by Pan  et al. \cite{PLC} that a normal bimagic square of even order $2u$ exists if and only if $u \geq 4$. 

Much less is known about general bimagic squares.  It is known that general bimagic squares of orders  3 and 4 are impossible but we do not know anything about $5 \times 5$ bimagic squares \cite{Bo2}.  In 2006 Wroblewski and Morgenstern independently constructed a limited number of numerical examples of general bimagic squares of orders 6 and 7  and a few other numerical examples of higher order general bimagic squares are also known \cite{Bo2}. Iyer constructed a parametrized family of $6 \times 6$ general pseudo-bimagic squares. We have not seen in the published literature any other parametrized family of magic squares of orders 6 or 8 with any bimagic properties.

In this paper we will construct a parametrized family of general pseudo-bimagic squares of order 6, and another parametrized family of general bimagic squares of order 8. We will use these families to construct, for any positive integer $m > 1$,  families of $6m \times 6m$ pseudo-bimagic squares as well as  families of $8m \times 8m$ bimagic squares whose entries are given in terms of a large number of parameters. We also give  numerical examples of $6 \times 6$ and $12 \times 12$ pseudo-bimagic squares and an $8 \times 8$ general bimagic square.

\section{Preliminary lemmas}.  
In this section we first give three  preliminary lemmas for the construction of squares with bimagic properties. Following the notation used in \cite{CL}, if $A$ is any matrix, we will write  $A^{*2}$ to denote the matrix obtained by replacing each entry of the matrix $A$ by its square. Thus, $A$ will be a bimagic square if both $A$ and $A^{*2}$ are magic squares. We will use the words `bimagic matrix' and `bimagic square' interchangeably. 

The first lemma gives a very simple way of constructing a general bimagic square using a given bimagic square. This result is not new --- we state it here as a lemma as we need it in the sequel. The second lemma shows how a bimagic matrix of order $2n$ may be constructed starting from two $n \times n$ semi-magic matrices that satisfy certain conditions. For writing the bimagic matrix succinctly, we need the anti-diagonal matrix. For any positive integer $n$, we will denote the $n \times n$ anti-diagonal matrix by $J_n$. The last lemma shows how, under certain conditions,  a bimagic matrix of order $mn$ may be constructed using $m^2$ bimagic matrices of order $n$.

\begin{lemma}\label{simplelemma} If $X=[x_{ij}]$ is an $n \times n$  semi-bimagic, pseudo-bimagic or  bimagic matrix with rational entries, and $a$ and $b$ are arbitrary nonzero rational numbers with $a \neq 0$, then the matrix $Y=[y_{ij}]=[ax_{ij}+b]$ is also an $n \times n$  semi-bimagic, pseudo-bimagic or  bimagic matrix with rational entries according as $X$ is semi-bimagic, pseudo-bimagic or bimagic. 
\end{lemma}
\begin{proof} The proof is straightforward --- the bimagic property of the matrix $Y$ follows immediately from the expansion of $(ax_{ij}+b)^2$. If the magic sums of the matrices $X$ and $X^{*2}$ are $s_1$ and $s_2$ respectively, the corresponding magic sums of the matrices $Y$ and $Y^{*2}$ are $as_1+nb$ and $a^2s_2+2abs_1+nb^2$, respectively.
\end{proof}

\begin{lemma}\label{bimagiclemma1} Let  $X=[x_{ij}]$ and $Y=[y_{ij}]$ be two $n \times n$ matrices such that  the two matrices  $X^{*2}$
and  $Y^{*2}$ are both semi-magic and have the same magic sum. 
 If the sums of the columns of the matrix $X$ are equal to the sums of the corresponding columns of the matrix $Y$, that is, for each $j, j=1, \ldots, n$, we have
\begin{equation}
\sum_{i=1}^nx_{ij}=  \sum_{i=1}^ny_{ij},
\label{condcolXY}
\end{equation}
then the $2n \times 2n$ matrix $M$ defined by
\begin{equation}
M=  \begin{bmatrix} X & -XJ_n \\ -J_nY & J_nYJ_n \end{bmatrix} \label{semibimagicM}
\end{equation}
is a semi-bimagic matrix. If the entries of the matrices $X$ and $Y$ also satisfy the condition,
\begin{equation}
\sum_{i=1}^n (x_{ii}+ y_{ii}) =0, \label{addlcond1XY} 
\end{equation}
then the matrix $M$ is a pseudo-bimagic matrix, and further, if the  entries of the matrices $X$ and $Y$ satisfy the additional condition,
\begin{equation}
\sum_{i=1}^n (x_{ii}^2+y_{ii}^2) =2\sum_{j=1}^n x_{1j}^2, \label{addlcond2XY}
\end{equation}
then the matrix $M$ is a bimagic matrix.
\end{lemma}
\begin{proof} 
Since $J_n$ is the anti-diagonal matrix of order $n$, the matrix $M$ defined by \eqref{semibimagicM} may be written explicitly as follows:
\begin{equation}
\begin{bmatrix}
x_{11} & x_{12} & \ldots & x_{1n} & -x_{1n} & -x_{1,n-1}  & \ldots & -x_{12} & -x_{11} \\
x_{21} & x_{22} & \ldots & x_{2n} & -x_{2n} & -x_{2,n-1}  &  \ldots & -x_{22} & -x_{21} \\
\vdots  & &  & & & & & & \\
x_{n1} & x_{n2} & \ldots & x_{nn} & -x_{nn} & -x_{n,n-1}   & \ldots & -x_{n2} & -x_{n1} \\
-y_{n1} & -y_{n2} & \ldots & -y_{nn} &   y_{nn} & y_{n, n-1} & \ldots & y_{n2} &  y_{n1} \\
\vdots & &  & & & & & & \\
-y_{21} & -y_{22} & \ldots & -y_{2n}  & y_{2n} & y_{2, n-1} & \ldots & y_{22} &  y_{21} \\
-y_{11} & -y_{12} & \ldots & -y_{1n} & y_{1n} & y_{1, n-1} & \ldots &  y_{12} & y_{11} 
\end{bmatrix}
\end{equation}

It is now readily observed that the sum of the entries of  each row of the matrix $M$ is 0. Further, in view of the relations \eqref{condcolXY}, the sum of the entries of each  column of the matrix $M$ is also 0. Thus, $M$ is a semi-magic matrix whose magic sum is 0.

Nest we consider the matrix $M^{*2}$. If the  magic sum of both the semi-magic matrices  $X^{*2}$
and  $Y^{*2}$ is $s$, the sum of the entries of each row and each column of the matrix $M$ is readily seen to be $2s$. It follows that the matrix $M$ is semi-bimagic. 

We now note that when the entries of the matrices $X$ and $Y$ also satisfy the condition \eqref{addlcond1XY}, the sum of the entries of both the  diagonals of the matrix $M$ is 0 (which is the magic sum of the semi-bimagic matrix $M$), and hence $M$ is a pseudo-bimagic matrix. Finally, when  the additional condition \eqref{addlcond2XY} is also  satisfied, the sum of the entries of each of   the two diagonals of the matrix $M^{*2}$ is the same and equal to $2s$ which is the magic sum of the matrix $M^{*2}$. Thus, $M$ is a bimagic matrix.  This proves the lemma.
\end{proof}

\begin{lemma}\label{bimagiclemma2} Let $U_{ij}, i=1, \ldots, m, j=1, \ldots, m$, be $n \times n$  matrices  that are all either semi-bimagic or pseudo-bimagic or bimagic. If the magic sums of all the matrices $U_{ij}$ are equal, and the  magic sums of all the matrices $U_{ij}^{*2}$ are also equal,  then the matrix $M=[U_{ij}]$ is a semi-bimagic,  pseudo-bimagic or bimagic matrix  according as the matrices $U_{ij}$ are semi-bimagic, pseudo-bimagic or bimagic. 
\end{lemma}
\begin{proof}
We note that the matrix $M$ is order $mn$. If $s_1$ and $s_2$ are  the common magic sums of the matrices $U_{ij}$ and $U_{ij}^{*2}$, respectively, it is readily seen that all the  rows and columns of the matrix $M$ have the same magic sum $ms_1$. Similarly, all the rows and columns of the  matrix $M^{*2}$ have the same magic sum $ms_2$. If the matrices $U_{ij}$ are all pseudo-bimagic or bimagic, the diagonals of the matrices $M$ and $M^{*2}$ are also readily seen to satisfy the corresponding magic properties. This proves the lemma.
\end{proof}

\section{Parametrized families of  pseudo-bimagic  squares \\ of order $6m$}\label{bimagic6} -
We will now construct,  in Section \ref{Secpseudobimagic66}, a four-parameter family of $6 \times 6$ pseudo-bimagic matrices, and we will use it in the next subsection to construct a multi-parameter family of $6m \times 6m$  pseudo-bimagic matrices where $m$ is any positive integer greater than 1.
\subsection{A family of $6 \times 6$ pseudo-bimagic squares} \label{Secpseudobimagic66}
Euler had given the  matrix $E$ defined by
\begin{equation}
E=\begin{bmatrix}
p^2+q^2-r^2-s^2    &       2   q   r+2   p   s    &       2   q   s-2   p   r  \\[0.05in]      2   q   r-2   p   s    &       p^2-q^2+r^2-s^2    &       2   p   q+2   r   s  \\[0.05in]      2   q   s+2   p   r    &       2   r   s-2   p   q    &       p^2-q^2-r^2+s^2
\end{bmatrix}
\end{equation}
where $p, q, r$, and $s$ are arbitrary parameters and such that the matrix  $E^{*2}$ is semi-magic with its magic  sum being $(p^2 + q^2 + r^2 + s^2)^2$. 

We begin our construction of $6 \times 6$ pseudo-bimagic matrices by first interchanging the first two columns of the matrix $E$ and in the resulting matrix, we interchange the second and third columns to get the  matrix  $E_1(p, q, r, s)$ given by
\begin{equation}
E_1(p, q, r, s)=\begin{bmatrix} 
2 p s+2 q r     &     -2 p r+2 q s     &     p^2+q^2-r^2-s^2   \\[0.05in]    p^2-q^2+r^2-s^2     &     2 p q+2 r s     &     -2 p s+2 q r   \\[0.05in]    -2 p q+2 r s     &     p^2-q^2-r^2+s^2     &     2 p r+2 q s
\end{bmatrix}
\label{euler33a}
\end{equation}
Since interchanging any two columns does not alter either the sums of the entries of the rows  or the columns, the matrix $E_1^{*2}(p, q, r, s)$ is also  semi-magic with the same magic sum $(p^2 + q^2 + r^2 + s^2)^2$.  

We now construct two matrices $X$ and $Y$ by replacing the parameters $p, q, r$, and $s$ in the matrix $E_1(p, q, r, s)$ first by new parameters $p_1, q_1, r_1$, and $s_1$, and then by parameters $p_2, q_2, r_2$, and $s_2$, that is, 
\begin{equation}
X = E_1(p_1, q_1, r_1, s_1), \quad Y= E_1(p_2, q_2, r_2, s_2).
\label{matricesXY33}
\end{equation}
The two  matrices $X^{*2}$ and $Y^{*2}$ are both semi-magic and their magic sums are $(p_1^2 + q_1^2 + r_1^2 + s_1^2)^2$ and $(p_2^2 + q_2^2 + r_2^2 + s_2^2)^2$, respectively.

We will now choose the parameters $p_i, q_i, r_i$, and $ s_i, i=1, 2$, such that the matrices $X$ and $Y$ satisfy the conditions \eqref{condcolXY} of Lemma \ref{bimagiclemma1}. Accordingly, we  impose the following three conditions that would make the sums of the entries of the three columns of the matrix $X$  equal to the corresponding sums for the matrix $Y$:
\begin{multline}
p_1^2 - 2p_1q_1 + 2p_1s_1 - q_1^2 + 2q_1r_1 + r_1^2 + 2r_1s_1 - s_1^2 \\
 = p_2^2 - 2p_2q_2 + 2p_2s_2 - q_2^2 + 2q_2r_2 + r_2^2 + 2r_2s_2 - s_2^2, 
\label{condcol1}
\end{multline}
\begin{multline}
p_1^2 + 2p_1q_1 - 2p_1r_1 - q_1^2 + 2q_1s_1 - r_1^2 + 2r_1s_1 + s_1^2 \\
 = p_2^2 + 2p_2q_2 - 2p_2r_2 - q_2^2 + 2q_2s_2 - r_2^2 + 2r_2s_2 + s_2^2, 
\label{condcol2}
\end{multline}
\begin{multline}
p_1^2 + 2p_1r_1 - 2p_1s_1 + q_1^2 + 2q_1r_1 + 2q_1s_1 - r_1^2 - s_1^2 \\
= p_2^2 + 2p_2r_2 - 2p_2s_2 + q_2^2 + 2q_2r_2 + 2q_2s_2 - r_2^2 - s_2^2.
\label{condcol3}
\end{multline}

We note that if we square both sides of each of the three equations \eqref{condcol1}, \eqref{condcol2}, and \eqref{condcol3}, and add them up, we get 
\begin{equation}
3(p_1^2 + q_1^2 + r_1^2 + s_1^2)^2 = 3(p_2^2 + q_2^2 + r_2^2 + s_2^2)^2.
\end{equation}
It follows that when Eqs. \eqref{condcol1}, \eqref{condcol2}, and \eqref{condcol3} are satisfied,  the magic sums of the matrices 
$X^{*2}$ and $Y^{*2}$ will automatically be equal.

We will now solve Eqs. \eqref{condcol1}, \eqref{condcol2}, and \eqref{condcol3} by writing
\begin{equation}
\begin{aligned}
p_1    & =   m_1e_1+m_2e_2,    &   p_2    & =   m_1e_1-m_2e_2,   &   q_1    & =   m_1f_1+m_2f_2,   &   q_2    & =   m_1f_1-m_2f_2,\\
r_1    & =   m_1g_1+m_2g_2,   &   r_2    & =   m_1g_1-m_2g_2,   &   s_1    & =   m_1h_1+m_2h_2,   &   s_2    & =   m_1h_1-m_2h_2,
\end{aligned}
\label{subspqrs}
\end{equation}
where $e_i, f_i, g_i, h_i$, and $m_i, i=1, 2$, are arbitrary parameters. On substituting these values in the three equations \eqref{condcol1}, \eqref{condcol2}, and \eqref{condcol3}, transposing all terms to the left-hand side and factoring out $4m_1 m_2$, we get the following three equations, respectively:
\begin{align}
(e_1 - f_1 + h_1)e_2 + (-e_1 - f_1 + g_1)f_2 + (f_1 + g_1 + h_1)g_2 + (e_1 + g_1 - h_1)h_2& =0, \label{condcol1a} \\
(e_1 + f_1 - g_1)e_2 + (e_1 - f_1 + h_1)f_2 + (-e_1 - g_1 + h_1)g_2 + (f_1 + g_1 + h_1)h_2 & =0, \label{condcol2a}\\
(e_1 + g_1 - h_1)e_2 + (f_1 + g_1 + h_1)f_2 + (e_1 + f_1 - g_1)g_2 + (-e_1 + f_1 - h_1)h_2 & =0,\label{condcol3a}
\end{align}

Since  Eqs. \eqref{condcol1a}, \eqref{condcol2a}, and \eqref{condcol3} may be considered as three homogeneous linear equations in the variables $e_2, f_2, g_2$, and $h_2$,   their complete solution is readily obtained and is given by
\begin{equation}
e_2 = -f_1 - g_1 - h_1, \quad f_2 = e_1 + g_1 - h_1,\quad g_2 = e_1 - f_1 + h_1, \quad h_2 = e_1 + f_1 - g_1.
\label{sole2f2}
\end{equation}
Thus, a solution of Eqs. \eqref{condcol1}, \eqref{condcol2}, and \eqref{condcol3} is given by \eqref{subspqrs} where the values of $e_2, f_2, g_2$, and $h_2$ are given by \eqref{sole2f2}. Using the relations \eqref{subspqrs} and \eqref{sole2f2}, the  values of the parameters  $p_i, q_i, r_i, s_i$, $i=1, 2$, and hence also the matrices $X$ and $Y$ may be expressed explicitly in terms of arbitrary parameters $e_1, f_1, g_1, h_1, m_1$, and $m_2$. The  matrices $X^{*2}$ and $Y^{*2}$ have the same  magic sum and  the conditions \eqref{condcolXY} of Lemma \ref{bimagiclemma1} are now satisfied. Thus, on applying Lemma \ref{bimagiclemma1}, we can already obtain a semi-bimagic matrix of order $6 \times 6$. 

We will now choose the parameters $e_1, f_1, g_1, h_1, m_1$, and $m_2$ such that the matrices $X$ and $Y$ satisfy the  condition \eqref{addlcond1XY} of Lemma \ref{bimagiclemma1}. This condition may be written as follows:
\begin{multline}
12e_1^2m_2^2 + (4m_1^2 - 12m_2^2)e_1f_1 + (4m_1^2 - 12m_2^2)e_1g_1 + (4m_1^2 - 12m_2^2)e_1h_1\\
 - 4f_1^2m_2^2 + (4m_1^2 + 4m_2^2)f_1g_1 + (4m_1^2 + 4m_2^2)f_1h_1 - 4g_1^2m_2^2 \\
+ (4m_1^2 + 4m_2^2)g_1h_1 - 4h_1^2m_2^2 =0.
\label{condefgh1}
\end{multline}
Since Eq. \eqref{condefgh1} may be considered as a homogeneous equation of degree 2 in the variables $e_1, f_1, g_1$, and $h_1$, we can find its complete solution if one solution can be found. While Eq. \eqref{condefgh1} does not have a nontrivial solution for arbitrary values of $m_1$ and $m_2$, we found, with some experimentation, that if we take
\begin{equation}
m_1 = t^2 - 3, \quad  m_2 = 2t, \label{trym}
\end{equation}
where $t$ is an arbitrary parameter, Eq. \eqref{condefgh1} has a simple solution $(e_1, f_1, g_1, h_1)=(-t^2-3, -2t, -2t, 4t)$. Thus, when $m_1, m_2$ are given by \eqref{trym}, the complete solution of Eq. \eqref{condefgh1} is readily found and is given by
\begin{equation}
\begin{aligned}
e_1   &  =   (t^2 + 3)(12t^2u_1^2 - 2t(t^2 + 3)u_1u_2 - 2t(t^2 + 3)u_1u_3\\
& \quad \quad  +4t^2u_2^2 - (t^4 - 2t^2 + 9)u_2u_3 + 4t^2u_3^2), \\
f_1   &  =   -24t^3u_1^2 - 2t(t^4 - 12t^3 - 18t^2 - 36t + 9)u_1u_2- 2t(t^4 - 18t^2 + 9)u_1u_3 \\
& \quad \quad  + (t^6 - 2t^5 - 15t^4 - 4t^3 - 45t^2 - 18t + 27)u_2^2\\
& \quad \quad  + (t^6 - 4t^5 - 15t^4 - 8t^3 - 45t^2 - 36t + 27)u_2u_3 + 8t^3u_3^2, \\
g_1   &  =   -24t^3u_1^2 - 2t(t^4 - 18t^2 + 9)u_1u_2  - 2t(t^4 - 12t^3 - 18t^2 - 36t + 9)u_1u_3  \\
& \quad \quad + 8t^3u_2^2+ (t^6 - 4t^5 - 15t^4 - 8t^3 - 45t^2 - 36t + 27)u_2u_3 \\
& \quad \quad + (t^6 - 2t^5 - 15t^4 - 4t^3 - 45t^2 - 18t + 27)u_3^2, \\
h_1   &  =   4t(12t^2u_1^2 + (t^4 - 18t^2 + 9)u_1u_2 + (t^4 - 18t^2 + 9)u_1u_3\\
& \quad \quad  - 4t^2u_2^2 + (t^4 - 2t^2 + 9)u_2u_3 - 4t^2u_3^2),
\end{aligned}
\label{valefgh1}
\end{equation}
where $u_1, u_2$, and $u_3$ are arbitrary parameters. 

When the values of the parameters $p_i, q_i, r_i$, and $s_i$ are defined by \eqref{subspqrs} where the values of $e_2, f_2, g_2, h_2, m_1$, and $m_2$ are given by \eqref{sole2f2} and \eqref{trym}, and $e_1, f_1, g_1$, and $h_1$ are given by \eqref{valefgh1}, the semi-magic matrices $X$ and $Y$ defined by \eqref{matricesXY33} have the same magic sum and they satisfy the conditions \eqref{condcolXY} and \eqref{addlcond1XY} stipulated in Lemma \ref{bimagiclemma1}. Using the relations \eqref{subspqrs}, \eqref{sole2f2},  \eqref{trym}, and \eqref{valefgh1}, the entries of the matrices $X$ and $Y$ may be expressed in terms of four arbitrary  parameters $t, u_1, u_2$, and $u_3$. Now on applying Lemma \ref{bimagiclemma1} to the matrices $X$ and $Y$, we get a pseudo-bimagic matrix $M$, in terms of arbitrary parameters $t, u_1, u_2$, and $u_3$, which may be written as follows:
\begin{equation}
M(t, u_1, u_2, u_3)=  \begin{bmatrix} X & -XJ_3 \\ -J_3Y & J_3YJ_3 \end{bmatrix} \label{pseudobimagicM}
\end{equation}

The magic sum of the pseudo-bimagic matrix $M(t, u_1, u_2, u_3)$  is 0. Further, since the magic sums of the semi-magic matrices $X$ and $Y$ are given by $p_i^2+q_i^2+r_i^2+s_i^2$, $i=1, 2$, respectively, and these sums are equal, the magic sum of the matrix $M^{*2}(t, u_1, u_2, u_3)$ may be written, in terms of the arbitrary parameters $t, u_1, u_2$, and $u_3$, as $2\phi^2(t, u_1, u_2, u_3)$ where the value of  $\phi(t, u_1, u_2, u_3)$ can be worked out. Since the matrices $X, Y$, and  the value of $\phi(t, u_1, u_2, u_3)$ are too cumbersome to write, we do not give them explicitly

We give below a one-parameter family of pseudo-bimagic matrices derived from the  general four-parameter family of pseudo-bimagic matrices obtained above. Taking $(u_1, u_2$, $u_3)=(1, 0, 0)$, the four-parameter matrices $X$ and $Y$ mentioned in \eqref{pseudobimagicM} may be expressed in terms of a single parameter $t$. In fact, all the entries of the matrices $X$ and $Y$ have a common factor $144t^4$ which can be factored out in view of Lemma \ref{simplelemma},  and we write the resulting matrices as $X(t)$ and $Y(t)$, respectively. Denoting the polynomial $\sum_{j=0}^8c_jt^{8-j}$ by $(c_0, c_1, \ldots, c_8)$, the matrices  $X(t)$ and $Y(t)$ may be written as follows:

\scriptsize
\begin{equation*}
X(t)=\begin{bmatrix} 
 12(0, 1, 0, -1, -24, -9, 24, 9, 0)   & -24(0, 0, 1, 7, -6, -6, -3, -9, 0)   & (1, 0, -36, 0, 54, -576, -36, 0, 81)\\
(1, 0, -36, 0, 54, 576, -36, 0, 81)   &  -24(0, 0, 1, -7, -6, 6, -3, 9, 0) &  -12(0, 1, 0, -1, 24, -9, -24, 9, 0)\\
24(0, 0, 1, 5, 6, -6, -15, 9, 0)      &  (1, 0, 36, 0, -378, 0, -252, 0, 81)  & 24(0, 0, 1, -5, 6, 6, -15, -9, 0)
\end{bmatrix},
\label{defX33}
\end{equation*}

\begin{equation*}
Y(t)=\begin{bmatrix} 
 4(0, 1, 8, -9, -72, -9, 0, 81, 0)    &  8(0, 1, 1, 6, 18, -63, -27, 0, 0)  & (1, 0, -4, -192, 54, 0, -324, 0, 81) \\
(1, 0, -4, 192, 54, 0, -324, 0, 81)   &   -8(0, 1, -1, 6, -18, -63, 27, 0, 0)  & -4(0, 1, -8, -9, 72, -9, 0, 81, 0) \\
8(0, 1, -5, -6, 18, 45, 27, 0, 0)     & (1, 0, -28, 0, -378, 0, 324, 0, 81)   & -8(0, 1, 5, -6, -18, 45, -27, 0, 0)
\end{bmatrix}.
\label{defY33}
\end{equation*}
\normalsize

It now follows that the matrix $M(t)$ defined by
\begin{equation}
M(t)=  \begin{bmatrix} X(t) & -X(t)J_3 \\ -J_3Y(t) & J_3Y(t)J_3 \end{bmatrix} \label{pseudobimagicMt}
\end{equation}
is a one-parameter family of $6 \times 6$ pseudo-bimagic matrices. The magic sums of the matrices $M(t)$ and $M^{*2}(t)$ are 0 and  
$2\{(t^4 + 30t^2 + 9)(t^2 + 3)^2\}^2$, respectively. 

As a numerical example, on taking $t=2$ in the matrix $M(t)$, we get the pseudo-bimagic matrix $M_1$ given by
\begin{equation}
M_1=\begin{bmatrix}  -2952  &   -2736  &   -5855  &   5855  &   2736  &   2952\\ 3361  &   4848  &   -3960  &   3960  &   -4848  &   -3361\\ 5520  &   -4415  &   -720  &   720  &   4415  &   -5520\\ -2976  &   6207  &   1760  &   -1760  &   -6207  &   2976\\ -5793  &   -3424  &   2280  &   -2280  &   3424  &   5793\\ 2840  &   -480  &   6495  &   -6495  &   480  &   -2840
\end{bmatrix}.
\label{defM166}
\end{equation}
The magic sums of the matrices $M_1$ and $M_1^{*2}$ are 0 and $100962050 = 2.7105^2$, respectively.

In view of Lemma \ref{simplelemma}, we now add the constant 6496 to each entry of the matrix $M_1$ and obtain the following pseudo-bimagic matrix $P$ whose entries are all positive integers:
\begin{equation}
P=\begin{bmatrix} 3544   &    3760   &    641   &    12351   &    9232   &    9448   \\ 
   9857   &    11344   &    2536   &    10456   &    1648   &    3135   \\  
		12016   &    2081   &    5776   &    7216   &    10911   &    976   \\ 
			3520   &    12703   &    8256   &    4736   &    289   &    9472   \\
					703   &    3072   &    8776   &    4216   &    9920   &    12289   \\
							9336   &    6016   &    12991   &    1   &    6976   &    3656
\end{bmatrix}.
\label{defM1}
\end{equation}
We note that the entries of the matrix $P$ are all distinct. Further, the magic sums of the matrices $P$ and $P^{*2}$ are 38976 and 354150146, respectively.

\subsection{A multi-parameter family of $6m \times 6m$ pseudo-bimagic squares} \label{Secpseudobimagicsquares6mx6m}
In Section \ref{Secpseudobimagic66}, we obtained a pseudo-bimagic matrix $M(t, u_1, u_2, u_3)$, defined by \eqref{pseudobimagicM}, in terms of four arbitrary parameters $t, u_1, u_2$, and $u_3$, such that the magic sums of the matrices $M$ and $M^{*2}$ are 0 and $2\phi^2(t, u_1, u_2, u_3)$, respectively. It now follows from Lemma \ref{simplelemma} that if we divide each entry of the matrix $M$ by $\phi(t, u_1, u_2, u_3)$, we get a pseudo-bimagic matrix $M_1(t, u_1, u_2, u_3) =M(t, u_1, u_2, u_3)/\phi(t, u_1, u_2, u_3)$ such that the magic sums of the matrices $M_1(t, u_1, u_2, u_3) $ and $M_1 ^{*2}(t, u_1, u_2, u_3)$ are 0 and $2$, respectively. By assigning different values to the parameters $t, u_1, u_2$, and $u_3$, we can obtain an arbitrary number of $6 \times 6$ pseudo-bimagic matrices with the same magic sum 0, and the same magic sum 2 when the entries of these matrices are replaced by their squares. To construct a pseudo-bimagic matrix of order $6m$ for any positive integer $m > 1$, we first construct $m^2$ such matrices and then apply Lemma \ref{bimagiclemma2}. Further, on suitably applying Lemma \ref{simplelemma}, we can obtain a pseudo-bimagic matrix whose entries are multivariate polynomials with integer coefficients. With a suitable choice of parameters, we can readily ensure that the entries of the $6m \times 6m$ pseudo-bimagic matrix are all distinct. 

A simpler parametrized family of pseudo-bimagic matrices of order $6m$ is obtained by using the one-parameter pseudo-bimagic matrix $M(t)$ defined by \eqref{pseudobimagicMt}. On dividing each entry of the matrix $M(t)$ by $(t^4 + 30t^2 + 9)(t^2 + 3)^2$, we get the $6 \times 6$ matrix $U(t)=M(t)/\{(t^4 + 30t^2 + 9)(t^2 + 3)^2\}$ which, in view of Lemma \ref{simplelemma} is a pseudo-bimagic matrix such that the magic sums of the matrices $U(t)$ and $U^{*2}(t)$ are 0 and 2, respectively. For any positive integer $m > 1$, we may now construct $m^2$ pseudo-bimagic matrices $U_{ij} =U(t_{ij}), i=1, \ldots, m,  j=1, \ldots, m$, where $t_{ij}, i=1, \ldots, m,  j=1, \ldots, m$, are arbitrary parameters. Since all the matrices $U_{ij}$ are of order 6, on applying Lemma \ref{bimagiclemma2}, we get the $6m \times 6m$ matrix $U =[U_{ij}]$ which is a pseudo-bimagic matrix such that the magic sums of the matrices $U$ and $U^{*2}$ are 0 and $2m$, respectively. As before, we may apply Lemma \ref{simplelemma} to get a pseudo-bimagic matrix whose entries are multivariate polynomials with integer coefficients. 

As a numerical example, we will now construct a $12 \times 12$ pseudo-bimagic matrix. We will use the simple one-parameter $6 \times 6$ pseudo-bimagic matrix $M(t)$, defined  by \eqref{pseudobimagicMt}, to construct four pseudo-bimagic matrices. We have already obtained a numerical $6 \times 6$ pseudo-bimagic matrix $M_1$ defined by \eqref{defM166} such that the magic sums of the matrices $M_1$ and $M_1^{*2}$ are 0 and $2.7105^2$, respectively. By assigning to $t$ the three values 5, 6, and 7, respectively, in the matrix $M(t)$ defined by  \eqref{pseudobimagicMt}, we get three matrices which we multiply by the scalars $1/128, 1/81$, and $1/128$, respectively, to get the following three $6 \times 6$ pseudo-bimagic matrices:
\begin{equation*}
\begin{aligned}
M_2 & =\begin{bmatrix}  
 5580  &      -6165  &      -1648  &      1648  &      6165  &      -5580  \\     -523  &      1740  &      -8280  &      8280  &      -1740  &      523  \\     6360  &      5552  &      765  &      -765  &      -5552  &      -6360  \\     75  &      2148  &      8200  &      -8200  &      -2148  &      -75  \\     -7452  &      3925  &      -960  &      960  &      -3925  &      7452  \\     -4040  &      -7200  &      1923  &      -1923  &      7200  &      4040
\end{bmatrix},
\end{aligned}
\end{equation*}
\medskip

\begin{equation*}
\begin{aligned}
M_3 & =\begin{bmatrix}  
35560  &      -27216  &      -687  &      687  &      27216  &      -35560  \\     2385  &      4240  &      -44520  &      44520  &      -4240  &      -2385  \\     27120  &      35313  &      4816  &      -4816  &      -35313  &      -27120  \\     -3360  &      1295  &      44640  &      -44640  &      -1295  &      3360  \\     -37585  &      24096  &      -3528  &      3528  &      -24096  &      37585  \\     -24120  &      -37728  &      -721  &      721  &      37728  &      24120
\end{bmatrix},
\end{aligned}
\end{equation*}
\medskip
\begin{equation*}
\begin{aligned}
M_4 & =\begin{bmatrix}  
 70056  &      -40992  &      11405  &      -11405  &      40992  &      -70056  \\     14492  &      2331  &      -80640  &      80640  &      -2331  &      -14492  \\     40005  &      70940  &      9240  &      -9240  &      -70940  &      -40005  \\     -12152  &       -12336  &      80115  &      -80115  &      12336  &      12152  \\     -67461  &      46452  &      -3080  &      3080  &      -46452  &      67461  \\     -44940  &      -66395  &      -17040  &      17040  &      66395  &      44940
\end{bmatrix}.
\end{aligned}
\end{equation*}
\medskip

We have already noted in Section \ref{Secpseudobimagic66} that the magic sums of the pseudo-bimagic matrix $M(t)$ and $M^{*2}$ are 0 and $2\{(t^4 + 30t^2 + 9)(t^2 + 3)^2\}^2$, respectively. Accordingly, the magic sums of the  matrices $M_2, M_3$, and $M_4$ are all 0 while the the magic sums of the  matrices $M_2^{*2}, M_3^{*2}$, and $M_4^{*2}$ are $2.8477^2, 2.44785^2$, and $2.81965^2$, respectively. We now construct four new pseudo-bimagic matrices $U_1=M_1/7105, U_2=M_2/8477, U_3=M_3/44785$, and $U_4 =M_4/81965$. The magic sums of all four matrices $U_i, i=1, \ldots, 4$, are 0, and the magic sums of all four matrices $U_i^{*2}, i=1, \ldots, 4$, are 2. Accordingly, we may apply Lemma \ref{bimagiclemma2} to construct the following $12 \times 12$ pseudo-bimagic matrix:
\begin{equation}
U= \begin{bmatrix} U_1 & U_2\\ U_3 & U_4 \end{bmatrix}
\label{defU12}
\end{equation}
The entries of the matrix $U$ are all distinct rational numbers. Further, the magic sums of the matrices $U$ and $U^{*2}$ are 0 and 4, respectively. We can now apply Lemma 1 and obtain a numerical example in positive integers of a $12 \times 12$ pseudo-bimagic matrix. As this matrix involves  integers consisting of 15 digits, we do not give it explicitly.

\section{Parametrized families of bimagic squares of \\order $8m$}\label{bimagic8}
As in Section \ref{Secpseudobimagic66}, we  will first  construct, in the next subsection, a  family of $8 \times 8$ bimagic matrices, and then  use it  to construct a multi-parameter family of $8m \times 8m$ bimagic matrices where $m$ is any positive integer greater than 1.

\subsection{A family of $8 \times 8$ bimagic squares}\label{bimagicsquares8x8}
We will  construct a parametrized $8 \times 8$ bimagic square by starting from the following $4 \times 4$ matrix constructed by Euler:

\small
\begin{equation*}
E=\begin{bmatrix} 
a p+b q+c r+d s  &  a r-b s-c p+d q  &  -a s-b r+c q+d p  &  a q-b p+c s-d r   \\[0.05in]
  -a q+b p+c s-d r  &  a s+b r+c q+d p  &  a r-b s+c p-d q  &  a p+b q-c r-d s   \\[0.05in] 
		a r+b s-c p-d q  &  -a p+b q-c r+d s  &  a q+b p+c s+d r  &  a s-b r-c q+d p   \\[0.05in]
			-a s+b r-c q+d p  &  -a q-b p+c s+d r  &  -a p+b q+c r-d s  &  a r+b s+c p+d q  
\end{bmatrix},
\label{Euler}
\end{equation*}
\normalsize
where $a, b, c, d, p, q, r$,  and $s$ are arbitrary parameters. Euler has shown that the matrix $E^{*2}$ is semi-magic and its magic sum  $\sigma$ is given by
\begin{equation}
\sigma = (p^2 + q^2 + r^2 + s^2)(a^2 + b^2 + c^2 + d^2). \label{csumsqE}
\end{equation}

We now  choose  parameters $a, b, c$, and $d$ such that  that the sums of the four columns of the matrix $E$ may be written as $t_1, t_2, t_3$, and $t_4$ respectively. Thus  the parameters $a, b, c$, and $d$ must satisfy  four linear conditions which may be written succinctly as follows:
\begin{equation}
\begin{bmatrix}
p - q + r - s   &   p + q + r + s   &   -p - q + r + s   &   p - q - r + s   \\[0.05in] -p - q + r + s   &   -p + q + r - s   &   -p + q - r + s   &   p + q + r + s   \\[0.05in] -p + q + r - s   &   p + q - r - s   &   p + q + r + s   &   p - q + r - s   \\[0.05in] p + q + r + s   &   -p + q - r + s   &   p - q - r + s   &   p + q - r - s
\end{bmatrix}
\begin{bmatrix} a    \\[0.05in]  b     \\[0.05in] c    \\[0.05in] d \end{bmatrix}=
\begin{bmatrix} t_1    \\[0.05in] t_2    \\[0.05in] t_3    \\[0.05in] t_4 \end{bmatrix} 
\label{matrixeqn1}
\end{equation}
On pre-multiplying both sides of Eq. \eqref{matrixeqn1} by the invertible matrix
\begin{equation}
\begin{bmatrix}
 1   &   1   &   1   &   1   \\[0.05in]    1   &   1   &   -1   &   -1
   \\[0.05in]   1   &   -1   &   1   &   -1   \\[0.05in]    1   &   -1   &   -1   &   1
\end{bmatrix}
\end{equation}
and then dividing by 4, we get the matrix equation
\begin{equation}
\begin{bmatrix}
r   &   q   &   s   &   p   \\[0.05in]   -q   &   r   &   -p   &   s
   \\[0.05in]    -s   &   p   &   r   &   -q   \\[0.05in]    p   &   s   &   -q   &   -r
\end{bmatrix}
\begin{bmatrix} a    \\[0.05in]  b     \\[0.05in] c    \\[0.05in] d \end{bmatrix}=
\begin{bmatrix} 
(t_1 + t_2 + t_3 + t_4)/4 \\[0.05in] (t_1 + t_2 - t_3 - t_4)/4 \\[0.05in] (t_1 - t_2 + t_3 - t_4)/4 \\[0.05in](t_1 - t_2 - t_3 + t_4)/4
\end{bmatrix}
\label{matrixeqn2}
 \end{equation}
On re-writing the column vector on the right-hand side of Eq. \eqref{matrixeqn2} as the column vector $\{u_1, u_2, u_3, u_4\}$, we get the following values of $t_i, i=1, \ldots, 4$:
\begin{equation}
\begin{aligned}
 t_1   & =  u_1 + u_2 + u_3 + u_4,\\
 t_2   & =  u_1 + u_2 - u_3 - u_4, \\
t_3   & =  u_1 - u_2 + u_3 - u_4, \\
t_4   & =  u_1 - u_2 - u_3 + u_4,
\end{aligned}
\label{valt}
\end{equation}
and the matrix equation \eqref{matrixeqn2} may be written as 
 \begin{equation}
\begin{bmatrix}
r   &   q   &   s   &   p   \\[0.05in]   -q   &   r   &   -p   &   s
   \\[0.05in]    -s   &   p   &   r   &   -q   \\[0.05in]    p   &   s   &   -q   &   -r
\end{bmatrix}
\begin{bmatrix} a    \\[0.05in]  b     \\[0.05in] c    \\[0.05in] d \end{bmatrix}=
\begin{bmatrix}
 u_1    \\[0.05in] u_2    \\[0.05in] u_3    \\[0.05in] u_4 \end{bmatrix}.
\label{matrixeqn3}
\end{equation}

On solving Eq. \eqref{matrixeqn3}, we get the following values for $a, b, c$, and $d$:
\begin{equation}
\begin{aligned}
 a    & =  (pu_4 - qu_2 + ru_1 - su_3)/(p^2 + q^2 + r^2 + s^2), \\
b    & =  (pu_3 + qu_1 + ru_2 + su_4)/(p^2 + q^2 + r^2 + s^2), \\
c    & =  -(pu_2 + qu_4 - ru_3 - su_1)/(p^2 + q^2 + r^2 + s^2), \\
d    & =  (pu_1 - qu_3 - ru_4 + su_2)/(p^2 + q^2 + r^2 + s^2).
\end{aligned}
\label{valab}
\end{equation}

We now substitute the values of $a, b, c$, and $d$ given by \eqref{valab} in the matrix $E$ to get a new matrix $E_1$ which is now in terms  of arbitrary parameters $p, q, r, s, u_1, u_2, u_3$ and $u_4$ and  to make this fact explicit, wherever necessary, we will denote the matrix $E_1$ as $E_1(p, q, r, s, u_1, u_2, u_3, u_4)$. Since $E^{*2}$ is a semi-magic matrix, the matrix  $E_1^{*2}$ is also semi-magic and its magic sum, obtained by  using the values of $a, b, c$, and $d$ given by \eqref{valab} and the relation \eqref{csumsqE}, is  $u_1^2 + u_2^2 + u_3^2 + u_4^2$.  Further, the sums of the four columns of the matrix $E_1$ are given by $t_1, t_2, t_3$, and $t_4$,  respectively, where the values of $t_i, i=1, \ldots, 4$, are given by \eqref{valt}. Thus,  the magic sum of the semi-magic matrix $E_1^{*2}$ and the sums of the four columns of the  matrix $E_1$  are all independent of the parameters $p, q, r$, and $s$.
 
We now construct two matrices $X$ and $Y$ by replacing the parameters $p, q, r, s$ in the matrix $E_1$ by  $p_1, q_1, r_1, s_1$ and  $p_2, q_2, r_2, s_2$ respectively, that is, 
\begin{equation}
X=E_1(p_1, q_1, r_1, s_1, u_1, u_2, u_3, u_4), \quad Y = E_1(p_2, q_2, r_2, s_2, u_1, u_2, u_3, u_4).
\label{matricesXY}
\end{equation}

Both the   matrices $X^{*2}$ and $Y^{*2}$ are semi-magic and have the same magic  sum $u_1^2 + u_2^2 + u_3^2 + u_4^2$. Further, since the sums of the columns of the matrix $E_1(p, q, r, s$, $u_1, u_2, u_3, u_4)$ are independent of the parameters $p, q, r$, and $s$, it follows that the sums of the four columns of the matrix $X$ are equal to the sums of the corresponding columns of the matrix $Y$. We have thus  constructed two $4 \times 4$ semi-magic matrices $X$ and $Y$ with the same magic sum and such that the conditions \eqref{condcolXY} stated in Lemma \ref{bimagiclemma1} are satisfied.   
Using these matrices $X$ and $Y$,  the $8 \times 8$ matrix $M$ defined  by \eqref{semibimagicM}  where $J_n=J_4$ is the $4 \times 4$ anti-diagonal matrix,  is already a semi-bimagic matrix with the  magic sums of the matrices $M$ and $M^{*2}$ being   0 and $2(u_1^2 + u_2^2 + u_3^2 + u_4^2)$, respectively.

In order to make $M$ a fully bimagic matrix, the  entries of the matrices $X$ and $Y$ must satisfy  the additional conditions \eqref{addlcond1XY}  and \eqref{addlcond2XY}. We now note that the sum of the entries of the principal diagonal of the matrix $E$ is $(p + q + r + s)(a + b + c + d)$ which becomes 0 on taking $s=-p-q-r$. Accordingly, we take $s_i=-p_i-q_i-r_i, i=1, 2$, and we then get $\sum_{i=1}^4x_{ii}=0$ and also $\sum_{i=1}^4y_{ii}=0$. Hence, the condition \eqref{addlcond1XY} is now satisfied.

It remains to choose the parameters such that the last condition \eqref{addlcond2XY} of Lemma \ref{bimagiclemma1} is also satisfied. This condition may be written as a quadratic equation in the variables $u_i, i=1, \ldots, 4$, with the coefficients being polynomials of degree 8 in the variables $p_i, q_i, r_i, i=1, 2$. While we can reduce this equation to a solvable quadratic equation in the variables $u_i, i=1, \ldots, 4$ for certain values of $p_i, q_i, r_i, i=1, 2$, expressed in parametric terms, the resulting equation and its solution are too cumbersome to write. We, therefore, restrict ourselves to giving below a relatively simpler parametric solution of Eq. \eqref{addlcond2XY} by taking the following numerical values of $p_i, q_i, r_i, i=1, 2$:
\begin{equation}
(p_1, q_1, r_1, p_2, q_2, r_2) =(6,  -1, -2, 3, 1, 2).  \label{valpqr}
\end{equation}
We now get
\begin{equation}
s_1 =-p_1   -q_1   -r_1 = -3, \quad s_2 =-p_2 -q_2 -r_2=-6, \label{vals12}
\end{equation}
and with  these values, Eq. \eqref{addlcond2XY} may be written as follows:
\begin{equation}
144u_1^2 - 87u_1u_2 - 243u_1u_3 - 460u_1u_4 - 144u_2^2 + 820u_2u_3 + 27u_2u_4 + 183u_3u_4=0. \label{condu88}
\end{equation}
Since both $u_3$ and $u_4$ occur only in degree 1 in Eq. \eqref{condu88}, this equation can be readily solved and its complete solution may be written as follows:
\begin{equation}
\begin{aligned}
u_1  & =f(460f - 27g - 183h), \\
 u_2  & = g(460f - 27g - 183h), \\ 
u_3  & = h(460f - 27g - 183h),\\
  u_4  & = 144f^2 - 87fg - 243fh - 144g^2 + 820gh,
\end{aligned}
\label{valu}
\end{equation}
where $f, g$, and $h$ are arbitrary parameters.

Thus, with the numerical values of $p_i, q_i, r_i, s_i, i=1, 2$, given by \eqref{valpqr} and \eqref{vals12}, and the values of $u_i, i=1, \ldots, 4$, given by \eqref{valu} in terms of arbitrary parameters $f, g$, and $h$, the matrices $X$ and $Y$ defined by \eqref{matricesXY} satisfy all the conditions of Lemma \ref{bimagiclemma1}. It now  follows from Lemma \ref{bimagiclemma1} that  the matrix $M$ defined, in terms of arbitrary parameters $f, g$, and $h$, by
\begin{equation}
M(f, g, h)=  \begin{bmatrix} X & -XJ_4 \\ -J_4Y & J_4YJ_4 \end{bmatrix} \label{bimagicM88}
\end{equation}
 is a fully bimagic matrix of order $8$ with the magic sums of the matrices $M(f, g, h)$ and $M^(f, g, h){*2}$ being 0 and $\phi(f, g, h)$, respectively,  where 
\begin{multline}
\phi(f, g, h) =2(33489h^4 - (168360f - 9882g)h^3 + (304138f^2 - 423360fg + 706618g^2)h^2\\
 - (238344f^3 - 288324f^2g + 241056fg^2 + 226278g^3)h \\
+ 232336f^4 - 49896f^3g + 178426f^2g^2 + 216fg^3 + 21465g^4) \label{bimagicsumM88} 
\end{multline}

Since it is cumbersome to write the matrices $X$ and $Y$  in terms of three parameters $f, g$, and $h$, we give below the matrices $X$ and $Y$ explicitly only in terms of the single parameter $h$ by taking $(f, g)=(0, 1)$ in the solution \eqref{valu}: 
\scriptsize
\begin{equation*}
X=\begin{bmatrix}
-2379 h^2+19762 h-5247 & -4209 h^2-30074 h+3339     &    -5307 h^2+20456 h-2961     &    5673 h^2+7288 h-2493 \\ -8601 h^2-10472 h+1557     &    2379 h^2+1544 h+441     &    -183 h^2-2086 h+3141     &    -2013 h^2+40622 h-6417 \\ 2013 h^2-12214 h+9     &    5673 h^2-28222 h+5517     &    -6771 h^2-28432 h+4167     &    -1281 h^2-3536 h+2421  \\  -183 h^2+33424 h-4869     &    5307 h^2+7952 h-3447     &    3111 h^2-23138 h+4203     &    6771 h^2+7126 h+639
\end{bmatrix},
\end{equation*}
and 
\begin{equation*}
Y=\begin{bmatrix}
-5673 h^2+4516 h-2655      &    -4209 h^2-26780 h+3825      &    -5307 h^2+23750 h-2475      &    2379 h^2+21562 h-5085     \\      1281 h^2-15602 h+585      &    5673 h^2-27490 h+5625      &    -6771 h^2-27700 h+4275      &    -2013 h^2-364 h+1845     \\      2013 h^2+38654 h-6795      &    2379 h^2-9070 h-1125      &    -183 h^2-12700 h+1575      &    8601 h^2-5372 h+1935     \\      -6771 h^2+2932 h+315      &    5307 h^2+14540 h-2475      &    3111 h^2-16550 h+5175      &    183 h^2+35674 h-4545
\end{bmatrix}.
\end{equation*}

\normalsize

As a numerical example, taking $h=1$ we get a bimagic matrix which, on adding the constant $33873$ to each entry of the matrix, yields the following $8 \times 8$ bimagic square $P$ consisting of distinct positive integers:
\[
P=\begin{bmatrix}
46009   &    2929   &    46061   &    44341   &    23405   &    21685   &    64817   &    21737     \\[0.05in]  16357   &    38237   &    34745   &    66065   &    1681   &    33001   &    29509   &    51389     \\[0.05in]  23681   &    16841   &    2837   &    31477   &    36269   &    64909   &    50905   &    44065     \\[0.05in]  62245   &    43685   &    18049   &    48409   &    19337   &    49697   &    24061   &    5501     \\[0.05in]  37397   &    16501   &    42137   &    2561   &    65185   &    25609   &    51245   &    30349     \\[0.05in]  1   &    41689   &    45181   &    28709   &    39037   &    22565   &    26057   &    67745     \\[0.05in]  47609   &    50065   &    64069   &    34405   &    33341   &    3677   &    17681   &    20137     \\[0.05in]  37685   &    61037   &    17905   &    15017   &    52729   &    49841   &    6709   &    30061
\end{bmatrix}. 
\]
The magic sums of the matrices $P$ and $P^{*2}$ are 270984 and 11904921032, respectively.

\subsection{A multi-parameter family of $8m \times 8m$ bimagic squares} \label{bimagicsquares8mx8m}
We will now construct a multi-parameter family of bimagic squares of order $8m$ by following a procedure similar to that used in Section \ref{Secpseudobimagicsquares6mx6m}. Accordingly, we will first construct a  bimagic matrix $M_1(f, g)$ in terms of two arbitrary parameters $f$ and $g$ such that the magic sums of the matrices $M_1(f,g)$ and $M_1^{*2}(f, g)$ are 0 and 2, respectively. 

We note that the magic sum \eqref{bimagicsumM88}
 of the matrix $M^{*2}$ defined by \eqref{bimagicM88} is given by a quartic function of $h$ in which the coefficient of $h^4$ is $2.33489=2.183^2$. Hence, following a method described by Fermat (as quoted by Dickson \cite[p. 639]{Di}), we can find a value of $h$ such that the magic sum \eqref{bimagicsumM88} becomes twice a perfect square. We thus get
\begin{multline}
h = (22559519772f^4 + 67072618944f^3g - 265560069708f^2g^2 + 562650703056fg^3\\
\quad \quad  - 495403914781g^4)(732(1049472f^3 + 133054434f^2g - 291355732fg^2\\
 + 60467877g^3)^{-1}.
\label{valh}
\end{multline}
and now the magic sum \eqref{bimagicsumM88} may be written as $2\phi_1^2(f, g)$,
where
\begin{multline}
\phi_(f, g)=(466184221685713535760f^8 - 2420499132106383708288f^7g \\
\quad \quad + 803894990537344341312f^6g^2- 22856568561860566396032f^5g^3 \\
\quad \quad+ 249368254976401818742152f^4g^4 - 681142397840009845341984f^3g^5\\
 \quad \quad+ 910236239570706317048992f^2g^6 - 693731546021888738027136fg^7\\
 \quad \quad+ 262837686159586636186213g^8)/(2928(1049472f^3 + 133054434f^2g \\
- 291355732fg^2 + 60467877g^3)^2), \label{defphifg}
\end{multline}

We now construct a new bimagic matrix $M_1(f, g)$ by substituting the value of $h$ given by \eqref{valh} in the bimagic matrix $M$ defined by \eqref{bimagicM88} and dividing each entry of the resulting matrix by $\phi_1(f, g)$. The entries of the matrix $M_1(f, g)$ will be rational functions of the parameters $f$ and $g$. In view of Lemma \ref{simplelemma}, the matrix $M_1(f, g)$ is bimagic and the magic sums of the matrices $M_1(f, g)$ and $M_1^{*2}(f, g)$ are 0 and 2, respectively. 

To construct a parametrized family of bimagic matrices  of order $8m$ for any positive integer $m > 1$, we replace the parameters $(f, g)$ in the matrix $M_1(f, g)$ by $m^2$ pairs of arbitrary parameters $(f_{ij}, g_{ij}), i=1, \ldots, m, j=1, \dots, m$, and thus obtain $m^2$ bimagic  matrices $M_1(f_{ij}, g_{ij})$ such that all the matrices $M_1(f_{ij}, g_{ij})$ have the same magic sum 0 and all the matrices $M_1^{*2}(f_{ij}, g_{ij})$ have the same magic sum 2. We then apply Lemma \ref{bimagiclemma2} to get a  bimagic square of order $8m$ whose entries are rational functions of the parameters $(f_{ij}, g_{ij}), i=1, \ldots, m, j=1, 1, \dots, m$. By suitably applying Lemma \ref{simplelemma}, we finally get a family of $8 \times 8$ bimagic squares whose entries are given by multivariate polynomials with integer coefficients.

The method yields numerical examples of bimagic  matrices of order $8m$ in which the entries are very large  integers. Accordingly, we do not give any numerical examples. 

\section{Concluding remarks}
In this paper we used families of semi-magic matrices of orders 3 and 4 to construct parametrized families of general magic squares of orders 6 and 8 with bimagic properties. We used these families to construct multi-parameter families of pseudo-bimagic and bimagic squares of orders $6m$ and $8m$, respectively. While we could not find a parametrized family of bimagic squares of order 6, since we have obtained a four-parameter family of pseudo-bimagic squares of order 6, it seems likely that a parametrized family of general bimagic squares of order 6 should exist. Further, as noted in the introduction, bimagic squares of orders less than 5 do not exist, but  there  remains the open problem of finding numerical examples and possibly a parametrized family of general bimagic squares of order 5. It would be quite a challenging problem to find such families of general bimagic squares of orders 5 and 6. We leave these problems for future investigations.

\noindent Ajai Choudhry, 13/4 A Clay Square, Lucknow - 226001, India.

\noindent E-mail address: ajaic203@yahoo.com

\end{document}